\documentclass{amsart}

\usepackage{amsmath,amssymb,amscd}
\usepackage{amsthm}
\usepackage{subfigure}
\usepackage{verbatim}
\usepackage{fancybox}
\usepackage{subfigure}
\usepackage[usenames,dvipsnames]{color}
\usepackage{sidecap}
\usepackage{caption}
\usepackage{ifthen}
\usepackage{textcomp}
\usepackage[all]{xy}
\usepackage{tikz}
\usetikzlibrary{intersections}

\def\Radius     {.4}
\def\Hole{.1}
\def\Ecc {-0.05}

\newcommand{\KP}[1]{
  \begin{tikzpicture}[baseline=-0.65ex,scale=0.5]
  #1
  \end{tikzpicture}}

\newcommand{\KPA}{%
  \KP{\filldraw[color=gray, fill=none, thick] circle (0.3);}%
}

\newcommand{\KPB}{%
  \KP{
    \draw[color=gray,thick] (-\Radius,\Radius) -- (\Radius,-\Radius);
    \draw[color=gray,thick] (-\Radius,-\Radius) -- (-\Hole,-\Hole);
    \draw[color=gray,thick] (\Hole,\Hole) -- (\Radius,\Radius);
  }}
  
\newcommand{\KPC}{%
  \KP{%
    \draw[color=gray,thick] (-\Radius,\Radius) .. controls (0,-\Ecc) .. (\Radius,\Radius);
    \draw[color=gray,thick] (-\Radius,-\Radius) .. controls (0,\Ecc) .. (\Radius,-\Radius);
  }}
  
\newcommand{\KPD}{%
  \KP{%
    \draw[color=gray,thick] (-\Radius,-\Radius) .. controls (\Ecc,0) .. (-\Radius,\Radius);
    \draw[color=gray,thick] (\Radius,-\Radius) .. controls (-\Ecc,0) .. (\Radius,\Radius);
  }}
  
\newcommand{\KRr}{%
  \KP{%
\draw[color=gray,thick] (-\Hole,-\Hole) .. controls (-\Radius,-\Radius) and (\Radius,-\Radius) .. (\Hole,-\Hole);
\draw[color=gray,thick] (-\Radius,\Radius)--(\Hole,-\Hole);
\draw[color=gray,thick] (\Hole,\Hole)--(\Radius,\Radius);
  }}

\newcommand{\KRl}{%
  \KP{%
\draw[color=gray,thick] (-\Hole,-\Hole) .. controls (-\Radius,-\Radius) and (\Radius,-\Radius) .. (\Hole,-\Hole);
\draw[color=gray,thick] (-\Radius,\Radius)--(-\Hole,\Hole);
\draw[color=gray,thick] (-\Hole,-\Hole)--(\Radius,\Radius);
  }}

\newcommand{\R}{{\mathbb R}}
\newcommand{\C}{{\mathbb C}}

\theoremstyle{plain}
\newtheorem{teo}{Theorem}[section]
\newtheorem{lema}[teo]{Lemma}
\newtheorem{cor}[teo]{Corollary}

\theoremstyle{definition}
\newtheorem{defi}{Definition}[section]

\theoremstyle{remark}

\begin{document}

\nocite{*}

\title{Negami's like splitting formula for the Jones polynomial}

\author{J.M. Burgos}
\address{Departamento de Matem\'aticas, Centro de Investigaci\'on y de Estudios Avanzados, Av. Instituto Polit\'ecnico Nacional 2508, Col. San Pedro Zacatenco, C.P. 07360 Ciudad de M\'exico, M\'exico}
\email{burgos@math.cinvestav.mx}


\begin{abstract}
Negami's splitting formula cannot be directly applied to get a Jones polynomial splitting formula for the contraction of certain planar graphs in the decomposition become non planar. Therefore, we build a Negami's like splitting formula from the scratch. Now, the new splitting matrix doesn't have the form of a Lindstr\"om matrix and it would be interesting to have similar results for it.
\end{abstract}

\maketitle

\section{Introduction}

Recently, we showed that Negami's splitting formula for the Tutte polynomial \cite{Negami} holds in the specialization $xy=1$ where the Jones polynomial is defined \cite{Burgos}.

A priori, following Theorems relating the Jones polynomial of a link with the Tutte polynomial of its associated signed planar graph \cite{Thistlethwaite}, \cite{Kauffman_signed}, it seems plausible that expressing each term of the Negami's splitting formula in terms of the corresponding Jones polynomials would give us a splitting formula for it. However, this is not the case.

Negami's splitting formula cannot be written in terms of the Jones polynomial for there are planar graphs whose contractions are not planar and do not correspond to any link diagram. As an example consider an alternating link $L$ whose associated planar graph $G$ can be separated into two planar subgraphs $G_1$ and $G_2$ sharing only four vertices $v_1, \ldots v_{4}$ such that their union is the whole graph and $G_1$ is the planar graph shown in Figure \ref{Example_Graph} \footnote{In terms of section \ref{Surgeries}, the full surgery $L_{1}^{full}$ corresponding to the graph $G_1$ in Figure \ref{Example_Graph} is the alternating 11-crossing knot $K11a100$ in the Hoste-Thistlethwaite Knot Table.}. One of the terms in the Negami's splitting formula requires the identification $v_1= v_3$ and $v_2=v_4$ of the vertices in $G_1$. However, by the Kuratowski's Theorem, this contraction of the graph $G_1$ is not planar for removing the edge $a$ and contracting the vertices gives the complete $K_5$ graph. In particular, there is no link associated to this particular term in the splitting.

It is necessary then to have a Negami's like splitting formula for the Jones polynomial. In contrary to the Negami's splitting formula, the argument above shows that the new splitting formula cannot have all of its terms indexed by the whole set of partitions.

\begin{figure}
\begin{center}
  \includegraphics[width=0.28\textwidth]{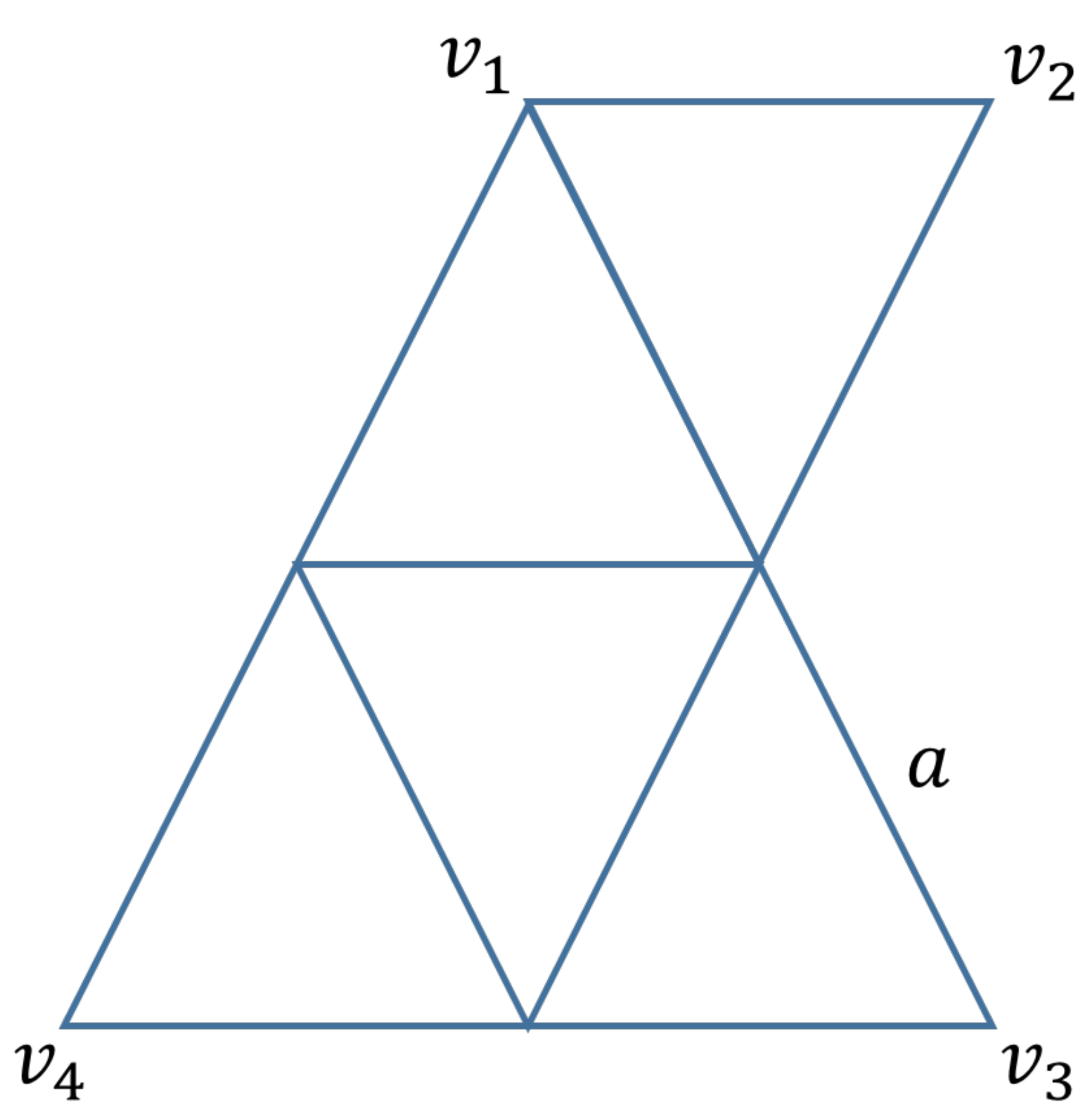}\\
  \end{center}
  \caption{An example of a planar graph where the contractions $v_1= v_3$ and $v_2=v_4$ give a non planar graph: Removing the edge $a$ gives the complete $K_5$ graph.}\label{Example_Graph}
\end{figure}

We say that a Jordan curve $C$ is an alternate cut of an oriented link $L$ if it is transversal to $L$ and walking along the curve in some direction the orientation of the $2n$ intersection points alternate. The cut $C$ separates the link diagram $L$ in two tangles $T_{1}$ and $T_{2}$ and all of the possible non crossing closures of these are the surgeries $L_{1}^{\mathcal{A}}$ and $L_{2}^{\mathcal{B}}$ respectively of the link diagram $L$, indexed by non crossing partitions $\mathcal{A}$ and $\mathcal{B}$. We say that the cut is non trivial if $n>0$. We denote by $\mathcal{A}\wedge\mathcal{B}$ the coarser partition finer than $\mathcal{A}$ and $\mathcal{B}$ and by $\mathcal{A}\vee\mathcal{B}$ the finer partition coarser than $\mathcal{A}$ and $\mathcal{B}$. We denote by $|S|$ the cardinal of a set $S$ and by $NC_n$ the semilattice of noncrossing partitions of $n$ elements.

The following is the main result of the paper:

\begin{teo}\label{Splitting_Jones_Intro}
Consider a non trivial alternate cut $C$ of a link $L$ with $2n$ intersection points. Then:
\begin{equation}\label{Formula_Intro}
J(L) = \sum_{\mathcal{A}, \mathcal{B}\in NC_n}\ c_{\mathcal{A} \mathcal{B}}\ J\left( L_{1}^{\mathcal{A}}\right)J\left( L_{2}^{\mathcal{B}}\right)
\end{equation}
where the matrix $\left(c_{\mathcal{A} \mathcal{B}}\right)$ is the inverse of the matrix $\left(d_{\mathcal{A} \mathcal{B}}\right)$ with entries:
\begin{equation}\label{Lindstrom_variant}
d_{\mathcal{A} \mathcal{B}}= (-t^{1/2}-t^{-1/2})^{n-|\mathcal{A}\wedge\mathcal{B}|+|\mathcal{A}\vee\mathcal{B}|-1}
\end{equation}
\end{teo}

It is interesting to compare the new splitting formula \eqref{Formula_Intro} with the one we naively would have expected directly from the Negami's splitting formula: Sum over the whole set of pairs of partitions (not just the non crossing ones) and replace $d_{\mathcal{A} \mathcal{B}}$ in Theorem \ref{Splitting_Jones_Intro} by the following:
\begin{equation}\label{Lindstrom}
d'_{\mathcal{A} \mathcal{B}}= (-t^{1/2}-t^{-1/2})^{|\mathcal{A}\vee\mathcal{B}|-1}
\end{equation}
The matrix \eqref{Lindstrom} has the form of a Lindstr\"om matrix \cite{Lindstrom}. Heuristically, the extra term $n-|\mathcal{A}\wedge\mathcal{B}|$ in the exponent is compensating the absence of the forbidden crossing partitions in the sum\footnote{The actual explanation of this term is the number of outer circles in a smoothing, Lemma \ref{Identity_Lemma}.}. The determinant of the matrix \eqref{Lindstrom} was calculated in \cite{Jackson} and \cite{BurgosII}. It would be very interesting to have similar results for the matrix \eqref{Lindstrom_variant}.

\begin{figure}
\begin{center}
  \includegraphics[width=0.4\textwidth]{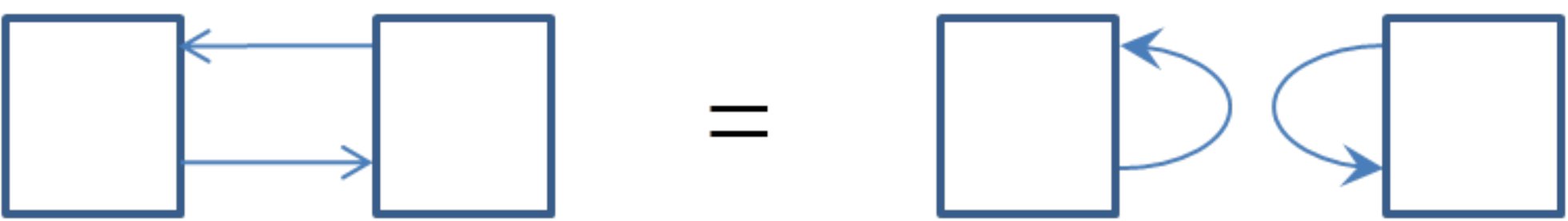}\\
  \end{center}
  \caption{Schematic picture of the Jones Polynomial factorization formula of a connected sum.}\label{Jones_Splitting_I}
\end{figure}

The case $n$ equals one reproduces the well known factorization of a connected sum:
$$J(L_{1}\# L_{2})=J(L_{1})J(L_{2})$$

The cases $n$ equals one, two and three are illustrated in Figures \ref{Jones_Splitting_I}, \ref{Jones_Splitting_II} and \ref{Jones_Splitting_III} respectively.

\begin{figure}
\begin{center}
  \includegraphics[width=.75\textwidth]{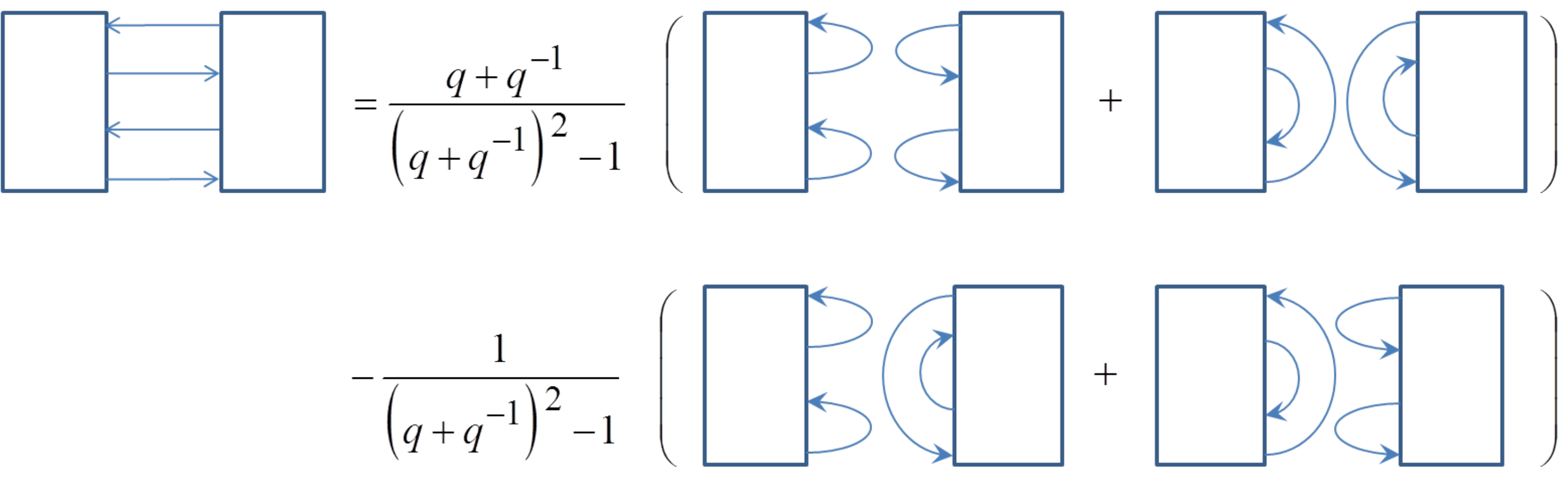}\\
  \end{center}
  \caption{Schematic picture of the Jones Polynomial splitting formula \eqref{SplittingII} in the case of four intersection points ($q= -t^{1/2}$).}\label{Jones_Splitting_II}
\end{figure}

\begin{figure}
\begin{flushleft}
  \includegraphics[width=1.1\textwidth]{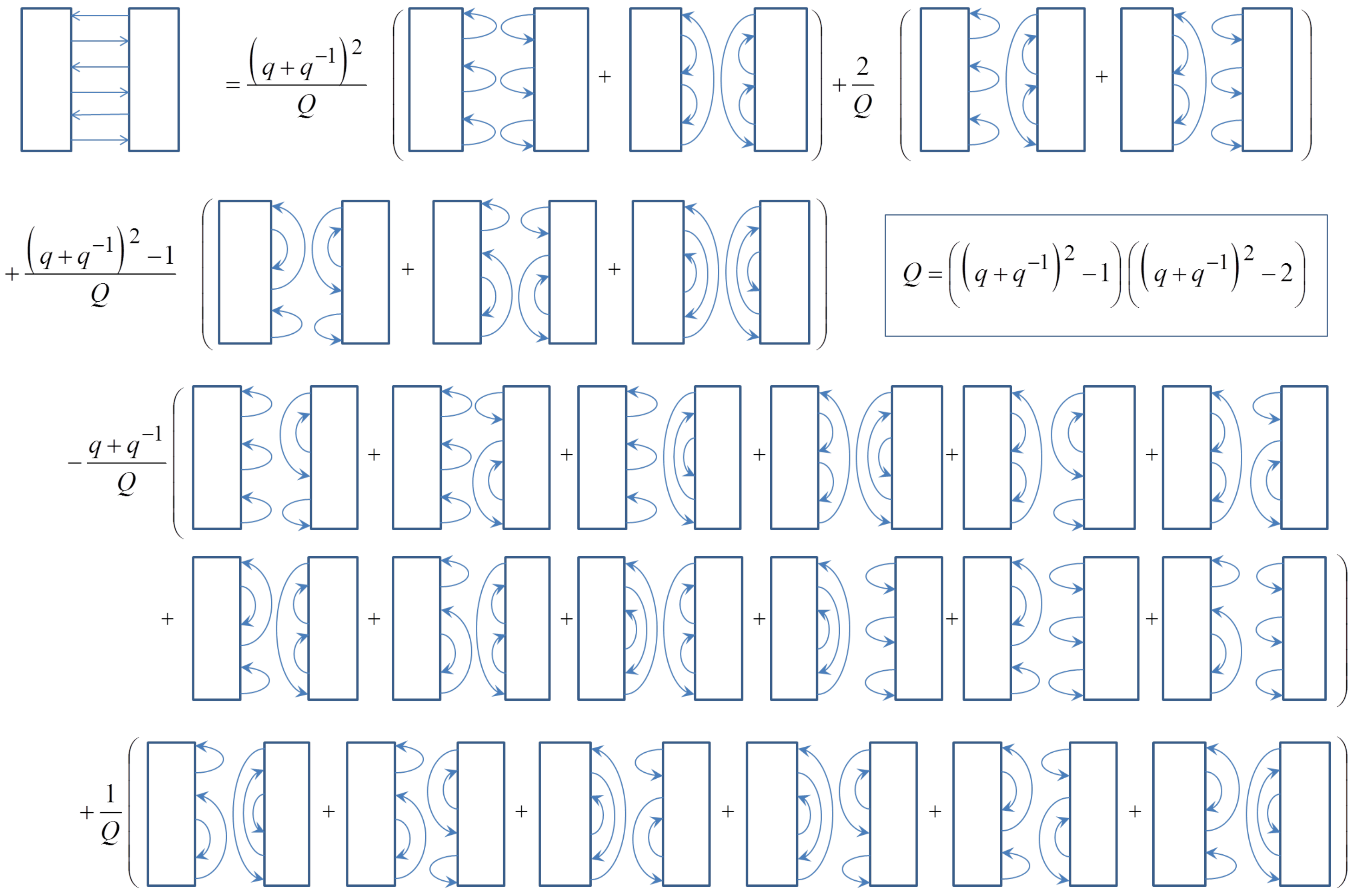}\\
  \end{flushleft}
  \caption{Schematic picture of the Jones Polynomial splitting formula in the case of six intersection points ($q= -t^{1/2}$).}\label{Jones_Splitting_III}
\end{figure}

Little knowledge is required to read this work and it is written for a broad audience. The categorified or Khovanov homology version of the result presented here, was written by the author in the unpublished paper \cite{BuIII}. However, the categorified version is in some sense a negative result for the calculation of the resulting spectral sequence requires the full knowledge of one of the complexes and not of the respective homology of it.

\section{Preliminaries: The Kauffman bracket and Jones polynomial}

The \textit{Kauffman bracket} of a Link diagram is a Laurent polynomial defined by the following rules:

\begin{enumerate}\label{Kauffman_Rules}
\item\label{Kauffman_Rule_One}
  $\left\langle\KPA\right\rangle=1$

\item\label{Kauffman_Rule_Two}
  $\left\langle L \sqcup \KPA\right\rangle=(-A^{2}-A^{-2})\langle L\rangle$

\item\label{Kauffman_Rule_Three}
  $\left\langle\KPB\right\rangle=
  A\left\langle\KPC\right\rangle + A^{-1} \left\langle \KPD \right\rangle$

\end{enumerate}
where the third rule denotes the surgery performed on a crossing and $\KPA$ is the unknot. To see that the bracket is well defined, it suffices to write it as a sum over all possible states. The set of crossings will be denoted by $\chi$ and a state is a vector with zero or one on its entries: $\alpha\in \{0,1\}^{\chi}$. Every state vector $\alpha$ defines a smoothing $\mathcal{S}_{\alpha}(L)$ of the link $L$ as follows: For each crossing $\KPB$, if $\alpha(\KPB)=0$, then perform the surgery with $\KPC$, otherwise perform the surgery with $\KPD$. The smoothing of a link is a disjoint union of circles. Denote by $k(\alpha)$ the number of these circles. The following is the sum state expression for the bracket (\cite{Kauffman}, Lemma 2.1):
\begin{equation}\label{State_expansion}
\left\langle L\right\rangle= \sum_{\alpha\in \{0,1\}^{\chi}}\ A^{|\alpha^{-1}(0)|-|\alpha^{-1}(1)|}(-A^{2}-A^{-2})^{k(\alpha)-1}
\end{equation}
where $|\cdot|$ denotes the cardinal of a set.

The second and third rule are defined in order to make the bracket invariant under the second and third Reidemeister moves (\cite{Kauffman}, Lemmas 2.2 and 2.3) while the first one is just a normalization. However, it is not invariant under the first Reidemeister move:

\begin{eqnarray*}
\left\langle\KRr\right\rangle &=& -A^{3}\left\langle
  \begin{tikzpicture}[baseline=-0.65ex,scale=0.5]
  \draw[color=gray,thick] (-\Radius, \Radius).. controls (0,-\Radius) ..(\Radius, \Radius);
  \end{tikzpicture}
\right\rangle \\
\left\langle\KRl\right\rangle &=& -A^{-3}\left\langle
  \begin{tikzpicture}[baseline=-0.65ex,scale=0.5]
  \draw[color=gray,thick] (-\Radius, \Radius).. controls (0,-\Radius) ..(\Radius, \Radius);
  \end{tikzpicture}
\right\rangle
\end{eqnarray*}

In order to solve this problem and have an ambient isotopic invariant, we must consider oriented link diagrams. For an oriented link diagram $L$, its \textit{writhe} $w(L)$ is defined as the sum of all of its sign crossings, see Figure \ref{Sign_Crossing}. See that in the case of a knot, its writhe does not depend on any particular orientation; i.e. The writhe is defined for unoriented knots. The \textit{Kauffman function} of an oriented link diagram is defined as follows:
\begin{equation}\label{Kauffman_function}
f_L(A):= (-A)^{-3w(L)}\left\langle L \right\rangle
\end{equation}
Now, in contrary to the bracket, the Kauffman function is an ambient isotopic invariant. The relation between the Kauffman function and the Jones polynomial\footnote{Actually, the Jones polynomial is a Laurent polynomial in $t^{1/4}$ and not a polynomial in $t$ in general.} is the following (\cite{Kauffman}, Lemma 2.6):
\begin{equation}\label{Jones_polynomial_def}
J(L)(t)=f_L(t^{-1/4})
\end{equation}
In what follows we will use the relation \eqref{Jones_polynomial_def} as a definition for the Jones polynomial.

\begin{figure}\label{Sign_Crossing}
\begin{center}
  \includegraphics[width=0.35\textwidth]{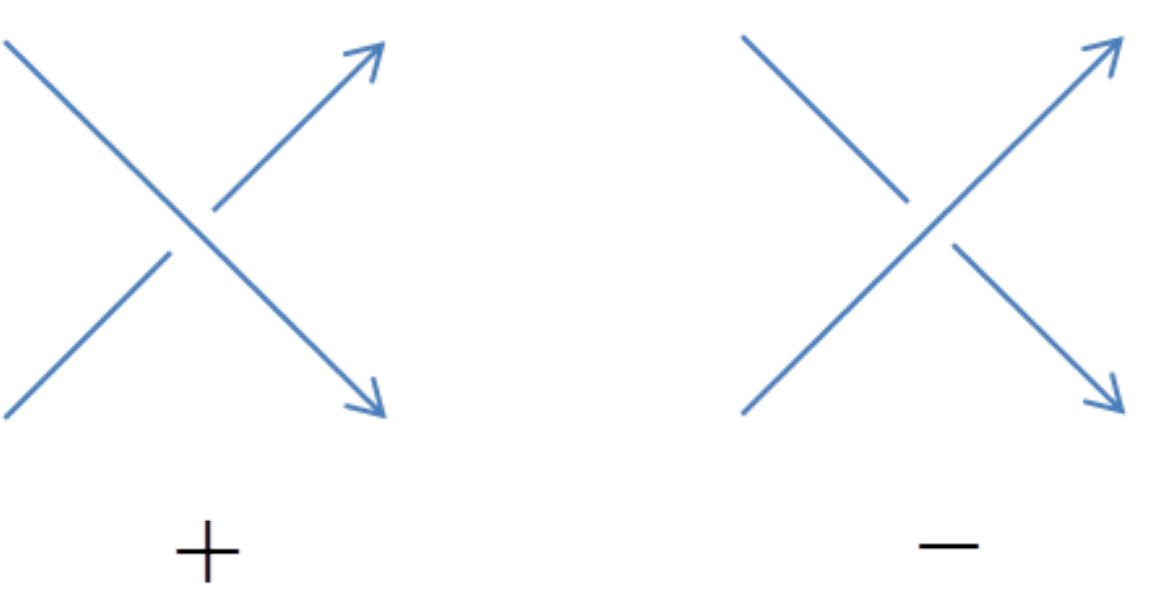}\\
  \end{center}
  \caption{Sign of a crossing.}\label{Sign_Crossing}
\end{figure}

\section{Alternate cuts and surgeries}\label{Surgeries}

Although this section is almost self-contained, we recommend the books \cite{Aigner} and \cite{Stanley}. Consider partitions $\mathcal{A}$ and $\mathcal{B}$ in $\Gamma_{n}$. We write $\mathcal{A}\prec\mathcal{B}$ if $\mathcal{A}$ is finer than $\mathcal{B}$; i.e. if for every $a\in\mathcal{A}$ there is $b\in\mathcal{B}$ such that $a\subset b$. We define the operations:

$$\mathcal{A}\vee \mathcal{B}= min\{\mathcal{C}\in \Gamma_{n}\ such\ that\ \mathcal{A}\prec \mathcal{C}\ and\ \mathcal{B}\prec \mathcal{C}\}$$
$$\mathcal{A}\wedge \mathcal{B}= max\{\mathcal{C}\in \Gamma_{n}\ such\ that\ \mathcal{A}\succ \mathcal{C}\ and\ \mathcal{B}\succ \mathcal{C}\}$$

The triple $(\Gamma_{n},\wedge, \vee)$ is a lattice with the following compatibility relations:
\begin{eqnarray*}
\mathcal{A}\wedge(\mathcal{B}\vee\mathcal{C})&=& (\mathcal{A}\wedge\mathcal{B})\vee(\mathcal{A}\wedge\mathcal{C}) \\
\mathcal{A}\vee(\mathcal{B}\wedge\mathcal{C})&=& (\mathcal{A}\vee\mathcal{B})\wedge(\mathcal{A}\vee\mathcal{C}) \\
\mathcal{A}\wedge full &=& full \\
\mathcal{A}\vee trivial &=& trivial
\end{eqnarray*}
for every triple of partitions $\mathcal{A}, \mathcal{B}$ and $\mathcal{C}$ where $trivial=\{\{1,2,\ldots n\}\}$ and $full= \{\{1\},\{2\},\ldots \{n\}\}$.

We define the notion of non crossing partitions as follows: Consider the k-th character $g_{k}:\R\rightarrow \C$ such that $g_{k}(t)=\exp(2\pi i t/k)$. For every partition $\mathcal{A}\in \Gamma_{n}$ such that $\mathcal{A}= \{m_{1},\ m_{2},\ldots m_{l}\}$ define:
$$Convex(\mathcal{A})=\{Convex\left(g_{n}(m_{1})\right),\ Convex\left(g_{n}(m_{2})\right),\ldots Convex\left(g_{n}(m_{l})\right)\}$$
where $Convex(S)$ denotes the convex hull of the set $S$ in $\C$. A partition $\mathcal{A}\in \Gamma_{n}$ will be called \textit{non crossing} \footnote{Non crossing partitions were introduced by Kreweras in \cite{Kreweras} and since then they have been widely used in different branches of mathematics \cite{mccammond}.} if for every pair of distinct convex sets in $Convex(\mathcal{A})$ their intersection is empty. The subset of non crossing partitions will be denoted by $NC_{n}$. The subset of non crossing partitions is closed under $\wedge$ but not under $\vee$; i.e. it is a sublattice of the semilattice $\Gamma_{n}^{\wedge}$ but not of $\Gamma_{n}^{\vee}$.

The permutation group $S_{n}$ acts on the set of $n$-partitions $\Gamma_{n}$ as follows: For every $n$-partition $\mathcal{A}= \{m_{1},\ m_{2},\ldots m_{l}\}$ we define:
$$\sigma\cdot\mathcal{A}:= \{\sigma(m_{1}),\ \sigma(m_{2}),\ldots \sigma(m_{l})\}$$
for every permutation $\sigma\in S_{n}$. See that the set of non crossing partitions $NC_{n}$ is closed under the action of the Dihedral group $D_{n}$ and this is the maximal permutation subgroup with this property.

\begin{defi}
An alternate cut of an oriented link diagram $L$ is a Jordan curve $C$ of the plane such that $C$ is transversal to the link diagram $L$ \footnote{In particular, the Jordan curve doesn't intersect any crossing of the link diagram $L$.} and, giving $C$ some orientation and walking the curve $C$ along this orientation, the intersection points orientation alternate.
\end{defi}

Consider an oriented alternate cut $C$ with a marked point $c\in C$. Because of degree theory, there must be $2n$ intersection point with $L$, half of them positively oriented and the other half negatively oriented. Denote these points by:
$$a_{1},b_{1},a_{2},b_{2},\ldots a_{n},b_{n}$$
as we go through the curve $C$ along its orientation starting in the marking $c$. We choose the marking in such a way that $a_{1}$ is positively oriented\footnote{This is achieved just moving the marking along the alternate cut.}. We will say that $\mathcal{A}$ is a non crossing partition of the intersection if $\mathcal{A}\in NC_n$. We will say that the cut is non trivial if $n>0$.

By considering the one point compactification of the plane, we may suppose that $C$ is an equator of the unit sphere in $\R^{3}$ and the intersection points:
$$a_{1},b_{1},\ldots a_{n},b_{n}$$
coincide with the points in the unit circle of the equator plane:
$$g_{2n}(1),g_{2n}(2),\ldots g_{2n}(2n-1),g_{2n}(2n)=1$$
respectively. Consider the projections of the respective hemispheres into the equator plane $\pi_N:H_1 \xrightarrow{\cong}\Delta$ and $\pi_S:H_2 \xrightarrow{\cong}\Delta$ where $\Delta$ is the closed unit disk.

We define the surgeries as follows: Given a non crossing partition $\mathcal{A}$ such that $\mathcal{A}= \{m_{1},m_{2},\ldots m_{l}\}$ and $i_{l,1}< i_{l,2}<\ldots i_{l,k_{l}}$ are all the elements in the class $m_{l}$, we define:
\begin{eqnarray*}
\vec{\mathcal{A}}:= \bigsqcup_{l} \overrightarrow{[g_{2n}(2i_{l,1}),g_{2n}(2i_{l,2}-1)]}\sqcup \overrightarrow{[g_{2n}(2i_{l,2}),g_{2n}(2i_{l,3}-1)]}\sqcup\ldots\\
\ldots\overrightarrow{[g_{2n}(2i_{l,k_{l}}),g_{2n}(2i_{l,1}-1)]}
\end{eqnarray*}
where $\overrightarrow{[b,a]}$ denotes the oriented line segment from the point $b$ to the point $a$ such that $a,b\in\R^{2}$. See that $\vec{\mathcal{A}}$ is a disjoint union of oriented segments starting at some $b_{i}$ and ending at some $a_{j}$. We define the oriented link diagram surgeries:
\begin{eqnarray}\label{surgeriesDef}
L_{1}^{\mathcal{A}} &:=& (L\cap H_1)\cup \pi_{S}^{-1}\left(\vec{\mathcal{A}}\right) \\
L_{2}^{\mathcal{A}} &:=& (L\cap H_2)\cup \pi_{N}^{-1}\left(\vec{\mathcal{A}}^{op}\right)
\end{eqnarray}
where $\vec{\mathcal{A}}^{op}$ is the diagram $\vec{\mathcal{A}}$ reversing all the vectors. For later purposes, we define the following non crossing oriented link: Consider non crossing partitions $\mathcal{A}$ and $\mathcal{B}$, then:

\begin{equation}\label{NoCrossingLink}
\vec{\mathcal{A}}*\vec{\mathcal{B}}^{op}:= \pi_N^{-1}\left(\vec{\mathcal{A}}\right)\cup\pi_S^{-1}\left(\vec{\mathcal{B}}^{op}\right)
\end{equation}
See Figure \ref{Example_Lemma} for an example.

\begin{figure}
\begin{center}
  \includegraphics[width=0.9\textwidth]{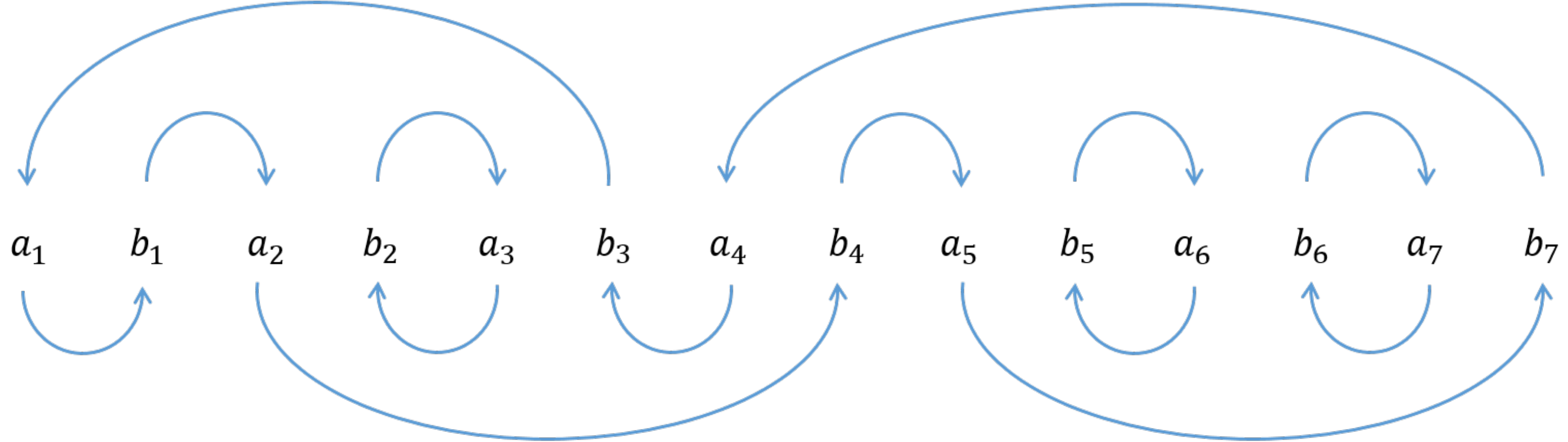}\\
  \end{center}
  \caption{Schematic picture of the non crossing oriented link $\vec{\mathcal{A}}*\vec{\mathcal{B}}^{op}$ with $\mathcal{A}= \{\{1,2,3\},\{4,5,6,7\}\}$ and $\mathcal{B}= \{\{1\},\{2,3,4\},\{5,6,7\}\}$. $\vec{\mathcal{A}}$ is shown on top and $\vec{\mathcal{B}}^{op}$ on bottom of the figure.}\label{Example_Lemma}
\end{figure}

Because the set of non crossing partitions is closed under the Dihedral group action, the set of these surgeries is independent of the orientation and marking of the Jordan curve $C$ previously chosen; i.e. Modulo ambient isotopy, the collection of surgeries only depends on the alternate cut as it was defined.

Because our construction neither modify the existing crossing nor adds new ones, we have that the number of positive/negative crossings $l_{1}^{\pm}$ is the same for all the surgeries $\{L_{1}^{\mathcal{A}}\}$ and a similar result holds for $\{L_{2}^{\mathcal{A}}\}$. Moreover,

\begin{equation}\label{relationSigns}
l^{\pm}= l_{1}^{\pm}+ l_{2}^{\pm}
\end{equation}
where $l^{\pm}$ denote the number of positive/negative crossings of the oriented link $L$. In other words, the number of crossings as well as the writhe are additive respect to the cut and this is independent of the surgery:
\begin{equation}\label{writhe}
w(L)= w\left(L_{1}^{\mathcal{A}}\right) + w\left(L_{2}^{\mathcal{B}}\right)
\end{equation}
for every pair of non crossing partitions $\mathcal{A}$ and $\mathcal{B}$.

The surgeries $L_{i}^{\mathcal{A}}$ have the common set of crossings $\chi_{i}:= \chi\cap H_{i}$ and the set of crossings $\chi$ of the oriented link $L$ is the disjoint union:
$$\chi=\chi_{1}\sqcup\chi_{2}$$
This way, the set of states decomposes as:
$$\{0,1\}^{\chi}= \{0,1\}^{\chi_{1}}\times \{0,1\}^{\chi_{2}}$$
In other words, every state $\alpha\in \{0,1\}^{\chi}$ gives a pair of states $\alpha_{i}\in \{0,1\}^{\chi_{i}}$ for $i=1,2$ such that $\alpha= (\alpha_{1},\alpha_{2})$.

\section{Negami's like splitting of the Jones polynomial}

From now on, we will denote simply by $L_{1}$ and $L_{2}$ the respective $full$-surgeries $L_{1}^{full}$ and $L_{2}^{full}$ where $full$ denotes the partition $\{\{1\},\{2\},\ldots \{n\}\}$. For $i=1,2$, we define the following maps:
$$\mathcal{C}_{i}:\{0,1\}^{\chi_{i}}\rightarrow NC_{n}$$
such that:
\begin{itemize}
\item If the set of crossings $\chi_{i}$ is empty, define the equivalence relation $\sim_{\emptyset}$ such that $i\sim_{\emptyset} j$ if $a_{i}$ and $a_{j}$ belong to the same connected component of the $full$-surgery $L_{i}$. We define:
$$\mathcal{C}_{i}(*)= \{1,2,\ldots n\}/\sim_{\emptyset}$$
where $*$ is the only state of $\{0,1\}^{\emptyset}$.

\item If the set of crossings $\chi_{i}$ is nonempty, for every state $\beta\in \{0,1\}^{\chi_{i}}$ define the equivalence relation $\sim_{\beta}$ such that $i\sim_{\beta} j$ if $a_{i}$ and $a_{j}$ belong to the same connected component of the smoothing $\mathcal{S}_{\beta}(L_{i})$. We define:
$$\mathcal{C}_{i}(\beta)= \{1,2,\ldots n\}/\sim_{\beta}$$
\end{itemize}
Because different connected components do not intersect, the obtained partition must be non crossing.

Given an alternate cut of a link diagram $L$, we say that a circle in the smoothing $\mathcal{S}_{\alpha}(L)$ is \textit{inner} if it doesn't contain any intersection point with the cut; i.e. it doesn't contain any of the points $a_{i},\ b_{j}$. Otherwise, the circle will be called \textit{outer}. The set of these circles will be denoted by $Inner_{\alpha}(L)$ and $Outer_{\alpha}(L)$ respectively and their disjoint union is the smoothing $\mathcal{S}_{\alpha}(L)$. We will say that an outer circle in a smoothing is a \textit{loop} if it contains two intersection points only\footnote{The transversality of the cut implies that there can only be an even number of intersection points on every circle.}.

\begin{lema}\label{Identity_Lemma}
Consider an alternate cut $C$ and a state $\alpha$ of a link $L$. Then:
$$|Outer_{\alpha}(L)|=n-|\mathcal{A}\wedge\mathcal{B}|+|\mathcal{A}\vee\mathcal{B}|$$
where $\mathcal{A}:= \mathcal{C}_{1}(\alpha_1)$ and $\mathcal{B}:= \mathcal{C}_{2}(\alpha_2)$.
\end{lema}
\begin{proof}
Following the definitions, it is clear that the set of circles $Outer_{\alpha}(L)$ is diffeomorphic to $\vec{\mathcal{A}}*\vec{\mathcal{B}}^{op}$. Abusing of notation, denote by $\vec{\mathcal{A}}\cap\vec{\mathcal{B}}$ the set of common vectors $\overrightarrow{[b_i,a_j]}$ in $\vec{\mathcal{A}}$ and $\vec{\mathcal{B}}$. There is a one to one canonical correspondence between $\vec{\mathcal{A}}\cap\vec{\mathcal{B}}$ and the set of loops in $\vec{\mathcal{A}}*\vec{\mathcal{B}}^{op}$.

Consider first the case of $|\mathcal{A}\vee\mathcal{B}|=1$. Then, $\vec{\mathcal{A}}*\vec{\mathcal{B}}^{op}$ consists of $|\vec{\mathcal{A}}\cap\vec{\mathcal{B}}|$ loops and one more circle containing all the remaining intersection points:
$$|\vec{\mathcal{A}}*\vec{\mathcal{B}}^{op}|=|\vec{\mathcal{A}}\cap\vec{\mathcal{B}}|+1$$
An easy application inclusion-exclusion principle leads to $|\vec{\mathcal{A}}\cap\vec{\mathcal{B}}|= n-|\mathcal{A}\wedge\mathcal{B}|$ and we have the result for the particular case (Recall the example in Figure \ref{Example_Lemma}).

In the general case, consider $\mathcal{A}\vee\mathcal{B}= \{m_{1},m_{2},\ldots m_{l}\}$ and denote by $\mathcal{A}_i$ the partition resulting from $\mathcal{A}$ by removing every element not contained in $m_i$. Denote by $n_i$ the number of elements in the class $m_i$. Then, applying the particular case to every class $m_i$ we have the result:
\begin{eqnarray*}
|\vec{\mathcal{A}}*\vec{\mathcal{B}}^{op}|&=&\sum_{i=1}^{l} |\vec{\mathcal{A}}_{i}*\vec{\mathcal{B}}_{i}^{op}| \\
&=& \sum_{i=1}^{l} \left( n_i-|\mathcal{A}_{i}\wedge\mathcal{B}_{i}|+1\right) \\
&=& \left(\sum_{i=1}^{l} n_i\right)-\left(\sum_{i=1}^{l}|\mathcal{A}_{i}\wedge\mathcal{B}_{i}|\right)+l \\
&=& n-|\mathcal{A}\wedge\mathcal{B}|+|\mathcal{A}\vee\mathcal{B}|
\end{eqnarray*}
\end{proof}

\begin{defi}
Consider an alternate cut $C$ of a link $L$ and a non crossing partition $\mathcal{A}$ of the intersection. Define:
$$\left\langle L_i\right\rangle_{\mathcal{A}}:= \sum_{\substack{ \alpha\in \{0,1\}^{\chi_{i}} \\ \mathcal{C}_i(\alpha)=\mathcal{A}}}\ A^{|\alpha^{-1}(0)|-|\alpha^{-1}(1)|}(-A^{2}-A^{-2})^{k(\alpha)-|\mathcal{A}|}$$
where $\chi_{i}$ is the set of crossings of $L_i$ (Recall that $L_i$ is the full surgery $L_i:= L_{i}^{full}$).
\end{defi}

\begin{lema}\label{Identity_One}
Consider a non trivial alternate cut $C$ of a link $L$ with $2n$ intersection points. Then:
$$\left\langle L\right\rangle = \sum_{\mathcal{A}, \mathcal{B}\in NC_n}\ a_{\mathcal{A} \mathcal{B}}\left\langle L_1\right\rangle_{\mathcal{A}}\left\langle L_2\right\rangle_{\mathcal{B}}$$
where the matrix $M:= \left(a_{\mathcal{A} \mathcal{B}}\right)$ is symmetric with Laurent polynomial entries is defined as follows:
$$a_{\mathcal{A} \mathcal{B}}= (-A^{2}-A^{-2})^{n-|\mathcal{A}\wedge\mathcal{B}|+|\mathcal{A}\vee\mathcal{B}|-1}$$
\end{lema}
\begin{proof}
For every state $\alpha=(\alpha_1, \alpha_2)$, we have the decomposition:
$$\mathcal{S}_{\alpha}(L)=Inner_{\alpha}(L)\sqcup Outer_{\alpha}(L)= Inner_{\alpha_1}(L_1)\sqcup Inner_{\alpha_2}(L_2)\sqcup Outer_{\alpha}(L)$$
Denote by $k(\alpha_i)$ the number of circles in the smoothing $\mathcal{S}_{\alpha_i}(L_i)$. By definition, the number of outer circles in the smoothing $\mathcal{S}_{\alpha_i}(L_i)$ is $|\mathcal{C}_{i}(\alpha_i)|$ and by Lemma \ref{Identity_Lemma} we have:

$$k(\alpha)= \underbrace{k(\alpha_{1})-|\mathcal{A}|}_\text{1.} + \underbrace{k(\alpha_{2})- |\mathcal{B}|}_\text{2.}+ \underbrace{n+|\mathcal{A}\vee\mathcal{B}|-|\mathcal{A}\wedge\mathcal{B}|}_\text{3.}$$
such that:
\begin{enumerate}
\item = Number of inner circles in the smoothing $\mathcal{S}_{\alpha_{1}}(L_{1})$. \\
\item = Number of inner circles in the smoothing $\mathcal{S}_{\alpha_{2}}(L_{2})$. \\
\item = Number of outer circles in the smoothing $\mathcal{S}_{\alpha}(L)$.
\end{enumerate}
where $\mathcal{A}:= \mathcal{C}_{1}(\alpha_1)$ and $\mathcal{B}:= \mathcal{C}_{2}(\alpha_2)$. Then, substituting in the sum state expression \eqref{State_expansion}:

\begin{eqnarray*}
\left\langle L\right\rangle &=& \sum_{\alpha\in \{0,1\}^{\chi}}\ A^{|\alpha^{-1}(0)|-|\alpha^{-1}(1)|}(-A^{2}-A^{-2})^{k(\alpha)-1} \\
&=&
\sum_{\mathcal{A}, \mathcal{B}\in NC_n}
\sum_{\substack{ \alpha_1\in \{0,1\}^{\chi_{1}} \\ \mathcal{C}_1(\alpha_1)=\mathcal{A}}}
\sum_{\substack{ \alpha_2\in \{0,1\}^{\chi_{2}} \\ \mathcal{C}_2(\alpha_2)=\mathcal{B}}}
A^{|\alpha_{1}^{-1}(0)|-|\alpha_{1}^{-1}(1)|}(-A^{2}-A^{-2})^{k(\alpha_{1})-|\mathcal{A}|}\ldots \\
&& A^{|\alpha_{2}^{-1}(0)|-|\alpha_{2}^{-1}(1)|}(-A^{2}-A^{-2})^{k(\alpha_{2})-|\mathcal{B}|}(-A^{2}-A^{-2})^{n+|\mathcal{A}\vee\mathcal{B}|-|\mathcal{A}\wedge\mathcal{B}|-1} \\
&=& \sum_{\mathcal{A}, \mathcal{B}\in NC_n} (-A^{2}-A^{-2})^{n+|\mathcal{A}\vee\mathcal{B}|-|\mathcal{A}\wedge\mathcal{B}|-1}\ldots \\
&& \left(\sum_{\substack{ \alpha_1\in \{0,1\}^{\chi_{1}} \\ \mathcal{C}_1(\alpha_1)=\mathcal{A}}}
A^{|\alpha_{1}^{-1}(0)|-|\alpha_{1}^{-1}(1)|}(-A^{2}-A^{-2})^{k(\alpha_{1})-|\mathcal{A}|}\right)\ldots \\
&& \left(\sum_{\substack{ \alpha_2\in \{0,1\}^{\chi_{2}} \\ \mathcal{C}_2(\alpha_2)=\mathcal{B}}}
A^{|\alpha_{2}^{-1}(0)|-|\alpha_{2}^{-1}(1)|}(-A^{2}-A^{-2})^{k(\alpha_{2})-|\mathcal{B}|}\right) \\
&=& \sum_{\mathcal{A}, \mathcal{B}\in NC_n}\ a_{\mathcal{A} \mathcal{B}}\left\langle L_1\right\rangle_{\mathcal{A}}\left\langle L_2\right\rangle_{\mathcal{B}}
\end{eqnarray*}
and the Lemma is proved.
\end{proof}

\begin{cor}\label{Identity_Two}
Under the hypothesis of Lemma \ref{Identity_One}, we have the following relation:
$$\left\langle L_{i}^{\mathcal{A}}\right\rangle = \sum_{\mathcal{B}\in NC_n}\ a_{\mathcal{A} \mathcal{B}}\left\langle L_i\right\rangle_{\mathcal{B}}$$
\end{cor}
\begin{proof}
We prove it for $L_1$ for the other case is verbatim. By the previous Lemma \ref{Identity_One}, we have:
$$\left\langle L_{1}^{\mathcal{A}}\right\rangle = \sum_{\mathcal{A}, \mathcal{B}\in NC_n}\ a_{\mathcal{B} \mathcal{C}}\left\langle L_1\right\rangle_{\mathcal{B}}\left\langle \vec{\mathcal{A}}*\overrightarrow{full}^{op}\right\rangle_{\mathcal{C}}$$
Because $\left\langle \vec{\mathcal{A}}*\overrightarrow{full}^{op}\right\rangle_{\mathcal{C}}$ equals the Kronecker delta $\delta_{\mathcal{A} \mathcal{C}}$ and the matrix $\left(a_{\mathcal{A} \mathcal{B}}\right)$ is symmetric, we have the result.
\end{proof}

See that $|\mathcal{A}\wedge\mathcal{B}|\geq|\mathcal{A}\vee\mathcal{B}|$ and equality holds if and only if $\mathcal{A}= \mathcal{B}$. In particular, respect to $(-A^{2}-A^{-2})$, the diagonal of $M$ is the maximum degree term of its determinant hence it is nonzero.

\begin{cor}\label{Splitting_Kauffman}
Under the hypothesis of Lemma \ref{Identity_One}, we have the following splitting formula for the Kauffman bracket:
$$\left\langle L\right\rangle = \sum_{\mathcal{A}, \mathcal{B}\in NC_n}\ b_{\mathcal{A} \mathcal{B}}\left\langle L_{1}^{\mathcal{A}}\right\rangle \left\langle L_{2}^{\mathcal{B}}\right\rangle$$
where the matrix $\left(b_{\mathcal{A} \mathcal{B}}\right)$ is the inverse of the matrix $M$.
\end{cor}
\begin{proof}
In matrix notation with column vectors, by Lemma \ref{Identity_One} we have:
$$\left\langle L\right\rangle= \left(\left\langle L_1\right\rangle_{\mathcal{A}}\right)^{T}M\left(\left\langle L_2\right\rangle_{\mathcal{A}}\right)$$
By Corollary \ref{Identity_Two}, $\left(\left\langle L_{i}^{\mathcal{A}}\right\rangle\right) = M.\left(\left\langle L_i\right\rangle_{\mathcal{A}}\right)$. Because the matrix $M$ is symmetric, we have the result:
$$\left\langle L\right\rangle = \left(\left\langle L_{1}^{\mathcal{A}}\right\rangle\right)^{T} \left(M^{-1}\right)^{T}M M^{-1}\left(\left\langle L_{2}^{\mathcal{A}}\right\rangle\right)= \left(\left\langle L_{1}^{\mathcal{A}}\right\rangle\right)^{T} M^{-1}\left(\left\langle L_{2}^{\mathcal{A}}\right\rangle\right)$$
\end{proof}

\begin{teo}\label{Splitting_Jones}
Consider a non trivial alternate cut $C$ of a link $L$ with $2n$ intersection points. Then:
\begin{equation}\label{Jones_Splitting_Formula}
J(L) = \sum_{\mathcal{A}, \mathcal{B}\in NC_n}\ c_{\mathcal{A} \mathcal{B}}\ J\left( L_{1}^{\mathcal{A}}\right)J\left( L_{2}^{\mathcal{B}}\right)
\end{equation}
where the matrix $\left(c_{\mathcal{A} \mathcal{B}}\right)$ is the inverse of the matrix $\left(d_{\mathcal{A} \mathcal{B}}\right)$ with entries:
$$d_{\mathcal{A} \mathcal{B}}= (-t^{1/2}-t^{-1/2})^{n-|\mathcal{A}\wedge\mathcal{B}|+|\mathcal{A}\vee\mathcal{B}|-1}$$
\end{teo}
\begin{proof}
Because the writhe is additive respect to the cut, relation \eqref{writhe}, by corollary \ref{Splitting_Kauffman} and the definition of the Kauffman function \eqref{Kauffman_function} and the Jones polynomial \eqref{Jones_polynomial_def}, we have the result.
\end{proof}


As an example, consider the case $n$ equals two. Then:

$$\left(d_{\mathcal{A} \mathcal{B}}\right)=
\begin{pmatrix}
-t^{1/2}-t^{-1/2} & 1 \\
1 & -t^{1/2}-t^{-1/2}
\end{pmatrix}$$
and the inverse matrix reads as follows:

$$\left(c_{\mathcal{A} \mathcal{B}}\right)=-\frac{1}{(t^{1/2}+t^{-1/2})^{2}-1}
\begin{pmatrix}
t^{1/2}+t^{-1/2} & 1 \\
1 & t^{1/2}+t^{-1/2}
\end{pmatrix}$$
This gives the following splitting formula:
\begin{eqnarray}\label{SplittingII}
J(L) &=& -\frac{t^{1/2}+t^{-1/2}}{\left(t^{1/2}+t^{-1/2}\right)^{2}-1}\left(J(L_{1})J(L_{2})+J(\widehat{L_{1}})J(\widehat{L_{2}})\right) \\
&&- \frac{1}{\left(t^{1/2}+t^{-1/2}\right)^{2}-1}\left(J(L_{1})J(\widehat{L_{2}})+J(\widehat{L_{1}})J(L_{2})\right) \nonumber
\end{eqnarray}
where $\widehat{L_{i}}$ denotes the trivial surgery. This formula is illustrated in Figure \ref{Jones_Splitting_II}.

\subsection*{Acknowledgment}
The author benefited from C\'atedras CONACYT program.

\end{document}